\documentclass[11pt]{article}
\usepackage[margin=1.3in]{geometry} 
\usepackage[dvipsnames]{xcolor}
\usepackage{amssymb,amsfonts,amsmath,bbm,mathrsfs,stmaryrd, mathtools,comment}
\usepackage{xcolor}
\usepackage{url}
\usepackage{dsfont}
\usepackage{enumerate}
\usepackage{enumitem}
\usepackage{tikz-cd}
\usepackage{multicol}
\usetikzlibrary{cd}
\usepackage{algorithm2e}

\usepackage[colorlinks,
             linkcolor=black!75!red,
             citecolor=blue,
             pdftitle={},
             pdfauthor={},
             pdfproducer={pdfLaTeX},
             pdfpagemode=None,
             bookmarksopen=true
             bookmarksnumbered=true]{hyperref}

\usepackage{tikz}
\usetikzlibrary{cd,arrows,calc,decorations.pathreplacing,decorations.markings,intersections,shapes.geometric,through,fit,shapes.symbols,positioning,decorations.pathmorphing}

\usepackage{braket}

\usepackage[amsmath,thmmarks,hyperref]{ntheorem}
\usepackage{cleveref}

\creflabelformat{enumi}{#2(#1)#3}

\crefname{section}{Section}{Sections}
\crefformat{section}{#2Section~#1#3} 
\Crefformat{section}{#2Section~#1#3} 

\crefname{subsection}{\S}{\S\S}
\crefformat{subsection}{#2\S#1#3} 
\Crefformat{subsection}{#2\S#1#3} 

\theoremstyle{plain}

\newtheorem{lemma}{Lemma}[section]
\newtheorem{proposition}[lemma]{Proposition}
\newtheorem{corollary}[lemma]{Corollary}
\newtheorem{theorem}[lemma]{Theorem}

\theoremstyle{nonumberplain}

\theoremstyle{plain}
\theorembodyfont{\upshape}
\theoremsymbol{\ensuremath{\blacklozenge}}

\newtheorem{definition}[lemma]{Definition}
\newtheorem{example}[lemma]{Example}

\crefname{definition}{definition}{definitions}
\crefformat{definition}{#2definition~#1#3} 

\crefformat{definition}{#2Definition~#1#3} 

\crefname{ex}{example}{examples}
\crefformat{example}{#2example~#1#3} 
\Crefformat{example}{#2Example~#1#3} 

\crefname{remark}{remark}{remarks}
\crefformat{remark}{#2remark~#1#3} 
\Crefformat{remark}{#2Remark~#1#3} 

\crefname{convention}{convention}{conventions}
\crefformat{convention}{#2convention~#1#3} 
\Crefformat{convention}{#2Convention~#1#3} 

\crefname{exercise}{exercise}{exercises}
\crefformat{exercise}{#2exercise~#1#3} 
\Crefformat{exercise}{#2Exercise~#1#3} 

\crefname{problem}{problem}{problems}
\crefformat{problem}{#2problem~#1#3} 
\Crefformat{problem}{#2Problem~#1#3} 

\crefname{lemma}{lemma}{lemmas}
\crefformat{lemma}{#2lemma~#1#3} 
\Crefformat{lemma}{#2Lemma~#1#3} 

\crefname{proposition}{proposition}{propositions}
\crefformat{proposition}{#2proposition~#1#3} 
\Crefformat{proposition}{#2Proposition~#1#3} 

\crefname{corollary}{corollary}{corollaries}
\crefformat{corollary}{#2corollary~#1#3} 
\Crefformat{corollary}{#2Corollary~#1#3} 

\crefname{theorem}{theorem}{theorems}
\crefformat{theorem}{#2theorem~#1#3} 
\Crefformat{theorem}{#2Theorem~#1#3} 

\crefname{assumption}{assumption}{Assumptions}
\crefformat{assumption}{#2assumption~#1#3} 
\Crefformat{assumption}{#2Assumption~#1#3} 

\crefformat{equation}{(#2#1#3)} 
\Crefformat{equation}{(#2#1#3)}

\theoremstyle{nonumberplain}
\theoremsymbol{\ensuremath{\blacksquare}}

\newtheorem{proof}{Proof.}

\DeclareMathOperator{\id}{id}

\newcommand\bC{{\mathbb C}}

\newcommand\bN{{\mathbb N}}

\newcommand\cC{{\mathcal C}}

\newcommand\cG{{\mathcal G}}

\newcommand\cO{{\mathcal O}}

\newcommand{\Tr}{\text{Tr}} 
\newcommand{\tr}{\text{tr}} 

\newcommand{\lp}{\left(}
\newcommand{\rp}{\right)}

\newcommand{\spn}[1]{\text{Span}\lp #1\rp} 

\begin{document}


\title{Local quantum Cuntz--Krieger algebras of dephasing quantum graphs}
\author{Lara Ismert\footnote{Department of Mathematics \& Statistics, University of New Mexico \hfill \url{ismert@unm.edu}}}
\date{\today}

\maketitle

\begin{abstract}
We establish a general structure theorem for local quantum Cuntz--Krieger families associated to finite quantum graphs whose relations are analogous to those of the partial isometries generating an ordinary graph $C^*$-algebra, with the quantum edge correspondence governing composition. We apply this structure theorem to the quantum dephasing graph on $M_n$ to obtain an explicit isomorphism between $\cO_n$ and the local quantum Cuntz--Krieger algebra. More generally, we obtain a unital embedding of $\cO_n$ in the local quantum Cuntz--Krieger algebra for the dephasing graph on a finite direct sum of $M_n$.  

\end{abstract}




\section{Introduction}

A graph $C^*$-algebra encodes the combinatorics and dynamics of a directed graph in its generators and relations (\cite{Kumjian04},\cite{BatesPaskRaeburnSzymanski2000}). 
Finite simple graphs, in particular, are in bijective correspondence with relations on finite sets. The notion of a {\em quantum relation} was formally introduced by Weaver in \cite{Weaver10} as a generalization of a classical relation. In \cite{DuanSeveriniWinter13}, Duan, Severini, and Winter introduced the notion of a {\em quantum confusability graph} for a quantum channel, which was subsequently shown by Weaver in \cite{Weaver21} to fit within the framework of quantum relations. The correspondence between classical relations on finite sets, finite simple graphs, and adjacency matrices led Brannan, Eifler, Voigt, and Weber to define a generalized notion of a quantum adjacency matrix in \cite{BEVW22}, providing yet another notion of a quantum graph. The benefit of viewing a finite simple graph as an adjacency matrix is two-fold: the quantum adjacency matrix is sensitive to the underlying quantum graph being directed, and, with an analogue of a $\{0,1\}$-matrix in hand, the authors could define an analogue of the Cuntz--Krieger algebra (\cite{CK80}) for a finite $\{0,1\}$-matrix, called a {\em quantum Cuntz--Krieger algebra}.

In \cite{BHINW23}, Brannan, Hamidi, Nelson, Wasilewski, and the present author introduced a unit condition to the original definition of a quantum Cuntz--Krieger algebra for a quantum graph $\cG$, and also defined and studied a quotient of a quantum Cuntz--Krieger algebra by so-called ``local relations.'' This quotient, called the {\em local quantum Cuntz--Krieger algebra for $\cG$}, was shown to be isomorphic to a Cuntz--Pimsner algebra for an associated $C^*$-correspondence, called the {\em quantum edge correspondence}, whenever $\cG$ has no quantum sources. Leveraging this identification, \cite[Section 4]{BHINW23} contains the only examples to date of the isomorphism class of a local quantum Cuntz--Krieger algebra. Moreover, the examples contained therein are extremal examples---rank-one quantum graphs (e.g., the trivial quantum graph) reflecting the smallest level of connectivity between vertices, and complete quantum graphs reflecting the highest level of connectivity between vertices. Only in the case of the complete quantum graph was the explicit isomorphism between the local quantum Cuntz--Krieger algebra given, and in this case, the local quantum Cuntz--Krieger algebra for the complete quantum graph on a finite-dimensional $C^*$-algebra $B$ was shown to be isomorphic to $\cO_{\dim B},$ the Cuntz algebra on $\dim B$ generators.

In general, working with the generators and relations of a (local) quantum Cuntz--Krieger algebra is challenging because the relations are given only in terms of a matrix containing the generators as coefficients, rather than relations directly between generators (as one has for Cuntz--Krieger algebras, and more generally, graph algebras). Moreover, calculating the isomorphism classes of local quantum Cuntz--Krieger algebras for non-extremal quantum graphs is of general interest to the quantum graph community due to the historical success of studying classical graphs via its associated graph $C^*$-algebra and current interest in the development of dynamical theory that serves as a quantum analogue to the infinite-path dynamics of classical graphs and subshifts. 

One of the main results of the present paper is Theorem~\ref{thm:lqck_structure_theorem}, a structure theorem for local quantum Cuntz–Krieger families. It gives orthogonality relations for the matrix coefficients of a local family and shows that their products are constrained by the associated quantum edge correspondence. In this sense, the theorem provides a quantum analogue of the structural relations satisfied by the edge partial isometries of an ordinary graph $C^*$-algebra: Let $E=(E^0,E^1,r,s)$ be a directed graph and $\{S_e:e\in E^1\}$ be a collection of partial isometries for a Cuntz--Krieger $E$-family. Then \cite[Proposition 1.12]{Rae05} states 
\begin{enumerate}
    \item $S_e^*S_f \neq 0 
    \quad \Longrightarrow \quad 
    e=f$;
    \item $S_e S_f\neq 0
    \quad \Longrightarrow \quad
    s(e)=r(f)$;
    \item $S_eS_f^*\neq 0 
    \quad  \Longrightarrow \quad s(e)=s(f)$.
    \end{enumerate}
When the graph $E$ is finite, the family $\{S_e:e\in E^1\}$ satisfies Cuntz--Krieger relations for the edge matrix of the graph $E$. Thinking now of $f_{ij}^{(a)}$ and $f_{k\ell}^{(b)}$ as elements of a finite quantum set, we obtain analogous relations on the quantum graph's local quantum Cuntz--Krieger algebra:\\

\noindent {\bf Theorem~\ref{thm:lqck_structure_theorem}.}    Let $\cG=(B,\psi,A)$ be a finite quantum graph such that $\psi$ is a $\delta$-form and $A$ is cp. Given $B\cong \oplus_{a=1}^d M_{n(a)}$, let $\{S^{(a)}_{ij}:1\leq i,j\leq n\}$ be the generators of a local quantum Cuntz--Krieger $\cG$-family $S$. Then 
\begin{enumerate}
    \item $(S_{ij}^{(a)})^* S_{k\ell}^{(b)} \neq 0 
    \quad \, \Longrightarrow \quad 
    a=b$ and $i=k$
    \item $S_{ij}^{(a)}S_{k\ell}^{(b)}\neq 0
    \quad \quad \,\,\,\Longrightarrow \quad
    e_{ij}^{(a)}\cdot \epsilon_{\cG} \cdot e_{k\ell}^{(b)}\neq 0$
    \item $S_{ij}^{(a)}(S_{k\ell}^{(b)})^*\neq 0 
    \quad \,\, \Longrightarrow \quad e_{ij}^{(a)}\cdot \epsilon_\cG\cdot A(e_{\ell m}^{(b)})\neq 0$ for some $1\leq m\leq n_b$,
    \end{enumerate}
where $\epsilon_\cG$ denotes the cyclic generator of the quantum edge correspondence for $\cG$.
For a generator $S_{ij}^{(a)}$, the column index $j$ behaves as a quantum source coordinate while the row index $i$ behaves as a quantum range coordinate; thus we can think of $ij$ as representing a quantum edge pointing from some sort of underlying vertices $j$ to $i$. The third relation replaces the classical condition $s(e)=r(f)$ that the two matrix units $e_{ij}^{(a)}$ and $e_{k\ell}^{(b)}$ have nonzero interaction in the quantum edge correspondence for $\cG$.

Theorem~\ref{thm:lqck_structure_theorem} is most useful when the quantum edge correspondence has a nice form. The dephasing quantum graph is one such example, and it is a natural example to consider. In \cite[Theorem 3.1]{Matsuda}, Matsuda identifies the dephasing graph on $M_2$ as one of only four (up to classical isomorphism) reflexive quantum graphs on $M_2$, with two of the other four being the trivial quantum graph and the complete quantum graph. As the isomorphism class of these extremal graphs were computed in \cite{BHINW23}, the remaining two quantum graphs (and their generalizations to higher dimensions), are canonical choices for which to compute the local quantum Cuntz--Krieger algebras.


In \cite{HIN25}, it was observed that the main result of \cite{Marrero-Muhly05} combined with \cite[Corollary 2.6]{BHINW23} allows one to identify the stable isomorphism class of a local quantum Cuntz--Krieger algebra for a single-vertex quantum graph. For the single-vertex dephasing quantum graph on $M_n$, the local quantum Cuntz--Krieger algebra is stably isomorphic to $\cO_n$. The main application of Theorem~\ref{thm:lqck_structure_theorem} is identifying Cuntz isometries in the local quantum Cuntz--Krieger algebra for the dephasing graph and showing that the Cuntz isometries generate the entire local quantum Cuntz--Krieger algebra in the case when the underlying vertex set is a single full matrix algebra $M_n$.

Section~\ref{sec:preliminaries} includes preliminaries on quantum graphs and (local) quantum Cuntz--Krieger algebras. For the sake of clarity in calculations for Theorem~\ref{thm:lqck_structure_theorem}, we introduce definitions of both the local and full quantum Cuntz--Krieger algebras, though the structure theorem itself pertains only to the local algebras. Section~\ref{sec:lqck_structure_theorem} is dedicated to proving Theorem~\ref{thm:lqck_structure_theorem}, a structure theorem for arbitrary local quantum Cuntz--Krieger families. In Section~\ref{sec:cuntz_algebra_embedding_theorem}, we apply Theorem~\ref{thm:lqck_structure_theorem} to the local quantum Cuntz--Krieger algebra for the dephasing quantum graph to identify Cuntz isometries. The main result of the final section is the following:\\

\noindent {\bf Corollary~\ref{cor:cuntz_isometries_for_L(M_n-tr-D}.} The local quantum Cuntz--Krieger algebra for the dephasing graph on $M_n$ is isomorphic to the Cuntz algebra $\cO_n$.
 
\section{Preliminaries}\label{sec:preliminaries}

A finite quantum set consists of a finite-dimensional $C^*$-algebra $B$ equipped with a state $\psi$. As $B$ is finite-dimensional, it decomposes into a direct sum of finite-dimensional simple $C^*$-algebras $B\cong \oplus_{a=1}^d M_{n(a)}$ for some $d\in \bN$. Throughout, we assume $\psi$ is a $\delta$-form, which means the associated density matrix $\rho_\psi = \oplus \rho^{(a)}$ on $B$ is invertible and $\Tr_a((\rho^{(a)})^{-1}) = \delta^2$ for all $1\leq a \leq d$ where $\Tr_a$ denotes the non-normalized trace on $M_{n(a)}$. 
For examples to keep in mind, the unique $\delta$-form on $M_n$ is the normalized trace $\tr=\frac{1}{n}\Tr$. In this case, $\delta^2 = n^2$, and we call $(M_n,\tr)$ a {\em single-vertex finite quantum graph}.
 For general $B\cong \oplus_{a=1}^d M_{n(a)}$, the unique tracial $\delta$-form on $B$ is given by 
    \[
    \tr(x):= \frac{1}{\dim{B}}\sum_{a=1}^d n_a \Tr_a(x),
    \]
    and in this case, $\delta^2 = \dim{B}$.

The multiplication map $m:B\otimes B\to B$ given by $x\otimes y\mapsto xy$ has a Hermitian adjoint $m^*$ with respect to the inner products induced on $B\otimes B$ and $B$ by $\psi\otimes \psi$ and $\psi$, respectively. The map $m^*:B\to B\otimes B$ is called the {\em comultiplication} on $B$. When $\psi$ is a $\delta$-form, the composition $m  m^*$ equals $ \delta^2 \cdot \id_B.$ The following co-associativity property holds and will be used later:
\begin{equation}\label{eq:co-associativity}
    (m^*\otimes \id)  m^* = (\id \otimes m^*)  m^*.
\end{equation}
Also, we have the so-called {\em snake equations}:
\begin{equation}\label{eq:Snake_Lemma}
    (\id \otimes m)  (m^*\otimes \id) = m^*  m = (m\otimes \id)  (\id \otimes m^*).
\end{equation}
Rather than including a full introduction of string diagrams for Fr\"obenius algebras and using these to prove our later structure theorems for local quantum Cuntz--Krieger algebras, we rely solely on doing our computations via expansion and reduction of function compositions involving comultiplication with (\ref{eq:co-associativity}) and (\ref{eq:Snake_Lemma}) above. We note also the following notational expressions of associativity of multiplication $\mu: \mathcal{C}\odot \mathcal{C}\to \mathcal{C}$ for any $C^*$-algebra $\mathcal{C}$:
\begin{equation}\label{eq:associativity1}
\mu   (\id \otimes \mu) 
= (\id \otimes \mu)  \mu
\end{equation}
\vspace{-.5cm}
\begin{equation}\label{eq:associativity2}
\mu   (\id \otimes \mu)   (\id \otimes \mu \otimes \id) 
= \mu  (\mu \otimes \mu) 
= \mu   (\id \otimes \mu)   (\id \otimes \id\otimes \mu).
\end{equation}


\begin{definition}\label{def:quantum_graph}
    A {\em quantum adjacency matrix} on a finite quantum set $(B,\psi)$ is a linear map $A:B\to B$ such that
\[
m  (A\otimes A)  m^* = \delta^2 A.
\]
A {\em finite quantum graph} is a triple $(B,\psi,A)$ such that $\psi$ is a $\delta$-form.
\end{definition}


We will always assume that $A$ is cp. When $A$ is a quantum adjacency matrix, \cite[Proposition 2.23]{Matsuda} states $A$ is cp if and only if $A$ is positive, if and only if $A$ is $*$-preserving.


\begin{definition}\label{def:lqck_family}
Let $\cC$ be a unital $C^*$-algebra equipped with $\mu:\cC\odot \cC\to \cC$ defined by $x\otimes y\mapsto xy$. Defined in \cite{BHINW23}, a linear map $S:B\to \cC$ is a {\em local quantum Cuntz--Krieger $(B,\psi,A)$-family} if and only if
\begin{equation}\label{eq:lqck1}
    \mu  (\mu\otimes \id)  (S\otimes S^*\otimes S)  (m^*\otimes \id)
     =\delta^{-2} S  m \tag{\text{LQCK1}}
\end{equation}
\vspace{-.5cm}
\begin{equation}\label{eq:lqck2}
\mu  (S^*\otimes S) = \delta^{-2}  \mu  (S\otimes S^*)  m^*  A  m\tag{\text{LQCK2}}
\end{equation}
\vspace{-.5cm}
\begin{equation}\label{eq:lqck3}
     \mu  (S\otimes S^*)  m^*(1_B)=\delta^{-2} 1_{\mathcal{C}} \tag{\text{\text{LQCK3}}}.
\end{equation}
\end{definition}


By linearity, the $C^*$-algebra generated by a local quantum Cuntz--Krieger $(B,\psi,A)$-family $S:B\to \mathcal{C}$ can be generated as a $C^*$-algebra by the images $\{S(f^{(a)}_{ij})\}_{i,j=1}^n$ of $S$ applied to the adapted matrix units for $B\cong \oplus_{a=1}^d M_{n_a}$, where for each $1\leq a \leq d$ and $1\leq i,j\leq n_a:$ \[f_{ij}^{(a)}=\psi(e_{ij}^{(a)})^{-1} e_{ij}^{(a)}.\]
Note that the product of any two adapted matrix units satisfies
\[
f_{ij}^{(a)}f_{k\ell}^{(b)} = \delta_{\stackrel{a=b}{j=k}} \psi(e_{jj}^{(a)})^{-1} f_{i\ell}^{(a)}.
\]

 \begin{definition}
   Let $\cG:=(B,\psi,A)$ be a finite quantum graph such that $\psi$ is a $\delta$-form and $A$ is cp. A local quantum Cuntz--Krieger $\cG$-family $s:B\to \mathcal{C}$ is {\em universal} if for any other local quantum Cuntz--Krieger $\cG$-family $S:B\to \mathcal{C}$ there is a $*$-homomorphism $\rho: C^*(s(B))\twoheadrightarrow C^*(S(B))$ satisfying $\rho \circ s=S$. The {\em local quantum Cuntz--Krieger algebra} $L\cO(\cG)$ is the $C^*$-algebra generated by the range of the universal local quantum Cuntz--Krieger $\cG$-family. 
 \end{definition}


Given a finite quantum graph $\cG:=(B,\psi,A)$, every local quantum Cuntz--Krieger $\cG$-family is readily seen to be a quantum Cuntz--Krieger $\cG$-family due to the relation $m  m^* = \delta^2 \id$.


\begin{definition}\label{def:qck_family}
Let $\cG:=(B,\psi,A)$ be a finite quantum graph such that $\psi$ is a $\delta$-form and $A$ is cp. Let $\cC$ be a unital $C^*$-algebra equipped with $\mu:\cC\odot \cC\to \cC$ defined by $x\otimes y\mapsto xy$. Defined in \cite{BEVW22} with the third condition added in \cite{BHINW23}, a linear map $S:B\to \cC$ is a {\em quantum Cuntz--Krieger $\cG$-family} if and only if
\begin{equation}\label{eq:qck1}
   \mu  (\mu\otimes \id)  (S\otimes S^*\otimes S)  (m^*\otimes \id)
      m^*   =S \tag{\text{QCK1}}
\end{equation}
\vspace{-.5cm}
\begin{equation}\label{eq:qck2}
\mu  (S^*\otimes S)  m^* =  \mu  (S\otimes S^*)  m^*  A\tag{\text{QCK2}}
\end{equation}
\vspace{-.5cm}
\begin{equation}\label{eq:qck3}     \mu  (S\otimes S^*)  m^*(1_B)=\delta^{-2} 1_{\mathcal{C}} \tag{QCK3}.\end{equation} \end{definition}


Given a finite quantum graph $\cG:=(B,\psi,A)$ such that $\psi$ is a $\delta$-form, define the {\em quantum edge checker} to be the element of $B\otimes B$ given by 
\[
\epsilon_{\cG}
:=
\delta^{-2} (\id\otimes A)  m^*(1_B).\]
When $\cG$ is a simple classical graph, $\epsilon_\cG$ is the characteristic function on the set of edges $E^1\subset E^0\times E^0$. For $x,y\in B$ and $\xi = \sum_{j=1}^k a_j\otimes b_j\in B\otimes B$, denote the usual left and right actions of $B$ on $B\otimes B$ by
\[
x\cdot \xi \cdot y := \sum_{j=1}^k (x a_j)\otimes (b_j y).
\]
The {\em quantum edge correspondence} for $\cG$ is the $B$-submodule of $B\otimes B$ generated by $\epsilon_\cG$:
\[
E_{\cG}:= \spn{\{x \cdot \epsilon_\cG \cdot y:x,y\in B\}},
\]
equipped with the right $B$-linear inner product defined on spanning elements by
\[
\langle x_1\cdot \epsilon_\cG\cdot y_1\;|\; x_2\cdot \epsilon_\cG \cdot y_2\rangle
:=
\delta^{-2}y_1^*A(x_1^*x_2)y_2 \quad \forall x_1,y_1,x_2,y_2\in B.
\]
The $B$-bimodule $E_\cG$ with this $B$-valued inner product is a $B$-correspondence; see \cite[Section 2]{BHINW23} for more details. Moreover, \cite[Proposition 2.6]{BHINW23} states that $E_\cG$ is isomorphic to $B\otimes_A B$ as $B$-correspondences (see \cite[Example 1.1, Corollary 2.6]{BHINW23}).

When no central summand of $B$ belongs to the kernel of $A$, \cite[Corollary 3.7]{BHINW23} states $L\cO(\cG)\cong \cO_{E_{\cG}}$, where $\cO_{E_\cG}$ is the Cuntz--Pimsner algebra for $E_\cG$. Because we will not require the theory of Cuntz--Pimsner algebras in the present paper, we will not review the definition or construction, but refer the reader to \cite{Katsura04}. In the setting where $\cG$ has a single vertex $M_n$, the main result of \cite{Marrero-Muhly05} gives the stable isomorphism class of $\cO_{E_\cG}$.


\begin{theorem}\label{thm:morita_equivalence}
Let $(M_n,\tr,A)$ be a single-vertex finite quantum graph such that $A$ is cp. Then $L\cO(M_n,\tr,A)$ is Morita equivalent to $\cO_{r(A)}$, where $r(A)$ is the Kraus rank of $A$.
\end{theorem}


\noindent Because $\cO_{r(A)}$ is $\sigma$-unital, we can further conclude that $L\cO(M_n,\tr,A)$ and $\cO_{r(A)}$ are stably isomorphic under the hypotheses of the above theorem. While Theorem~\ref{thm:morita_equivalence} is quite useful for identifying a candidate isomorphism class of $L\cO(M_n,\tr,A)$, it does not indicate the exact isomorphism class---more relations amongst the generators of $L\cO(M_n,\tr,A)$ seem to be necessary.


\section{Structure theorem for local quantum Cuntz--Krieger algebras}\label{sec:lqck_structure_theorem}

Throughout, $\cG=(B,\psi,A)$ is a finite quantum graph such that $\psi$ is a $\delta$-form and $A$ is cp, and $E_\cG$ is the associated quantum edge correspondence generated cyclically by $\epsilon_\cG$. 
Throughout, $\mathcal{C}$ denotes a unital $C^*$-algebra and $\mu:\mathcal{C}\odot \mathcal{C}\to \mathcal{C}$ is defined by $x\odot y\mapsto xy$. The proof of the following structure theorem iteratively uses the local quantum Cuntz--Krieger relations along with associativity and co-associativity of $\mu$ and $m^*$, respectively.

\begin{theorem}\label{thm:lqck_structure_theorem}
    Suppose $B\cong \oplus_{a=1}^d M_{n_a}(\bC)$, and let $\cG:=(B,\psi,A)$ be a quantum graph such that $\psi$ is a $\delta$-form and $A$ is cp. Given any local quantum Cuntz--Krieger $\cG$-family $S:B\to \mathcal{C}$, for all $1\leq a,b\leq d$ and $1\leq i,j\leq n_a$, $1\leq k,\ell\leq n_b$: 
    \begin{enumerate}
        \item $(S^{(a)}_{ij})^* S^{(b)}_{k\ell} \neq 0 
        \quad \Longrightarrow \quad 
        a=b\text{ and }i=k$.
        \item $S^{(a)}_{ij} S^{(b)}_{k\ell} \neq 0 
        \quad \quad \; \Longrightarrow \quad
        f_{ij}^{(a)}\cdot \epsilon_\cG\cdot f_{k\ell}^{(b)}\neq 0$.
        \item $S^{(a)}_{ij} (S^{(b)}_{k\ell})^* \neq 0 
        \quad \,\Longrightarrow \quad
        \exists 1\leq m\leq n_b$ such that $f_{ij}^{(a)}\cdot \epsilon_\cG\cdot A(f_{\ell m}^{(b)}) \neq 0$. 
    \end{enumerate}
\end{theorem}


\begin{proof}
    Let $\cG = (B,\psi,A)$ be a finite quantum graph and let $S:B\to \mathcal{C}$ be a local quantum Cuntz--Krieger $\cG$-family. The first property follows immediately by (\ref{eq:lqck2}).

{\em Proof of $(2)$.} Observe
    \begin{align*}
        \mu  (S\otimes S)
        &\overset{(\ref{eq:qck1})}{=}
        \mu  ([\mu   (\mu \otimes \id)   (S\otimes S^*\otimes S)  (m^*\otimes \id)  m^*]\otimes S) \\
        &\overset{(\ref{eq:co-associativity})}{=}
        \mu  ([\mu   (\mu \otimes \id)   (S\otimes S^*\otimes S)  ( \id\otimes m^*)  m^*]\otimes S) \\
        &= \mu   (\mu\otimes \id)  \circ [ S\otimes(\mu   (S^*\otimes S)   m^* ) \otimes S]  \circ (m^*\otimes \id) \\
        &\overset{(\ref{eq:qck2})}{=} \mu   (\mu\otimes \id) \circ   [ S\otimes(\mu   (S\otimes S^*)   m^*   A) \otimes S]   \circ(m^*\otimes \id) \\
        &= \mu   (\mu\otimes \mu)  \circ [ S\otimes S\otimes S^* \otimes S]  \circ
        (\id \otimes m^* \otimes \id)  \circ [(\id \otimes A)  m^* \otimes \id] \\
        &= \mu   (\mu\otimes \id) \circ  [ S\otimes S\otimes (\mu  (S^* \otimes S))] \circ (\id \otimes m^* \otimes \id)  \circ [(\id \otimes A)  m^* \otimes \id]\\
        &\overset{(\ref{eq:lqck2})}{=} \mu   (\mu\otimes \id) \circ   [ S\otimes S\otimes \delta^{-2}(\mu  (S \otimes S^*)  m^*  A  m)]  \\
        & \quad  \circ (\id \otimes m^* \otimes \id)  \circ [(\id \otimes A)  m^*\otimes \id] \\
       &= \mu   (\mu\otimes \id)  
       \circ
       [ S\otimes S\otimes (\mu  (S \otimes S^*)  m^*  A)] \\
        & \quad \quad  \circ (\id \otimes \id \otimes m) \circ   (\id \otimes m^* \otimes \id)  \circ  [\delta^{-2}(\id \otimes A)  m^* \otimes \id].
    \end{align*}
    We focus on the expression in the last line of the above computation. Equation~\ref{eq:Snake_Lemma} allows us to rewrite
    \[
    (\id \otimes \id \otimes m)\circ (\id\otimes m^*\otimes \id)
    =
    \id \otimes m^*m
    =
    (\id\otimes m^*)\circ (\id \otimes m),
    \]
    so we can say that $\mu(S\otimes S) = \mathcal{W}\circ (\id \otimes m)\circ [\delta^{-2}(\id \otimes A)m^*\otimes \id]$ for some linear operator $\mathcal{W}$. 
    As $\delta^{-2}(\id \otimes A)m^*(f_{ij}^{(a)})=f_{ij}^{(a)}\cdot \epsilon_\cG$, we have
    \[
    \left[(\id \otimes m) \circ [\delta^{-2}(\id\otimes A)  m^* \otimes \id]\right](f^{(a)}_{ij}\otimes f^{(b)}_{k\ell}) 
    = (\id \otimes m)(f^{(a)}_{ij}\cdot \epsilon_\cG\otimes f^{(b)}_{k\ell}) 
    = f^{(a)}_{ij}\cdot \epsilon_{\cG}\cdot f^{(b)}_{k\ell}.
    \]
Therefore, $f^{(a)}_{ij}\cdot \epsilon_{\cG}\cdot f^{(a)}_{k\ell}=0$ implies $S^{(a)}_{ij}S^{(b)}_{k\ell}=0$ for all $1\leq a,b\leq d$ and $1\leq i,j\leq n_a,1\leq k,\ell\leq n_b.$ We follow a similar strategy for obtaining the third claim. 

{\em Proof of $(3)$.} Observe
    \begin{align*}
    \mu  (S\otimes S^*)
    &\overset{(\ref{eq:qck1})}{=}
    \mu ([\mu  (\mu\otimes \id)   (S\otimes S^*\otimes S)   (m^*\otimes \id)  m^* ]\otimes S^*)\\   &\overset{(\ref{eq:associativity2})}{=}
    \mu (\mu\otimes \id) \circ  [S\otimes (\mu  (S^*\otimes S)  m^*) \otimes S^*) \circ (m^*\otimes \id)\\  
    &\overset{(\ref{eq:lqck2})}{=}
    \mu (\mu\otimes \id)  \circ [S\otimes (\mu  (S\otimes S^*)  m^*  A] \otimes S^*]\circ (m^*\otimes \id)\\   &\overset{(\ref{eq:associativity2})}{=}
    \mu (\mu\otimes \mu) \circ (S\otimes  S\otimes S^*\otimes S^*)  \circ [(\id\otimes m^*) \circ 
   ((\id\otimes A)  m^*\otimes \id)]\\
    &=
     \mu (\mu\otimes \mu) \circ (S\otimes  S\otimes S^*\otimes \id)  \circ \underbrace{[(\id\otimes m^*)(\id\otimes A)m^*\otimes S^*]}_{(\star)}.
    \end{align*}
Focusing on the $S^*$ factor in $(\star)$, applications of (QCK1) and (QCK2) yield
 \begin{align*}
     S^*
     &\overset{(\ref{eq:qck1})}{=}
      \mu   (\mu\otimes \id)   (S^*\otimes S \otimes S^*)  (m^*\otimes \id)  m^*\\
    &\,\,=
     \,\, \mu 
    [\mu  (S^*\otimes S)  m^* \otimes S^*] m^*\\
     &\overset{(\ref{eq:qck2})}{=}
 \mu  [\mu  (S\otimes S^*)  m^* A\otimes S^*] m^*\\
      &=
 \mu  [\mu  (S\otimes S^*)  m^* \otimes S^*] (A\otimes \id)m^*.
 \end{align*}
    Plugging this expression for $S^*$ back into $(\star)$, we have
    \begin{align*}
    (\star)
    &=
    [(\id\otimes m^*)(\id\otimes A)m^*] \otimes [\mu  [\mu  (S\otimes S^*)  m^* \otimes S^*] (A\otimes \id)m^*]\\
    &=
    (\id \otimes \mu  [\mu  (S\otimes S^*) \otimes S^*])
    \circ
    \underbrace{( \id\otimes m^*)(\id\otimes A)m^* \otimes (m^*\otimes \id)(A\otimes \id)m^*}_{(\star\star)}.
    \end{align*}

Our last task is to produce the map $(\id\otimes\id\otimes m\otimes \id\otimes \id)$ from the remaining terms of $\mu(S\otimes S^*)$ to post-compose with $(\star\star)$:
    \begin{align*}
        \mu(S\otimes S^*)
        &=
        \mu (\mu\otimes \mu) \circ (S\otimes  S\otimes S^*\otimes \id)  \circ (\star)\\
        &=\mu (\mu\otimes \mu) \circ (S\otimes  S\otimes S^*\otimes \id)\circ 
            (\id \otimes \mu  [\mu  (S\otimes S^*) \otimes S^*])
    \circ (\star\star)\\
    &= \mu(\mu\otimes \mu)(\mu\otimes \mu\otimes \mu) (S\otimes S\otimes S^*\otimes S\otimes S^*\otimes S^*)\circ 
    (\star\star)\\
     &= \mu(\mu\otimes \mu)(\mu\otimes \id \otimes \mu) (S\otimes S\otimes \mu(S^*\otimes S)\otimes S^*\otimes S^*)\circ 
    (\star\star)\\
    &\overset{(\ref{eq:lqck2})}{=} \mu(\mu\otimes \mu)(\mu\otimes \id \otimes \mu) (S\otimes S\otimes (\mu(S\otimes S^*)m^*A)\otimes S^*\otimes S^*)\\
       &\quad \quad  \circ (\id \otimes \id\otimes m\otimes \id\otimes \id)\circ 
    \delta^{-2}(\star\star).
    \end{align*}    
    We have shown $\mu(S\otimes S^*) = \mathcal{T}_1 \circ \mathcal{T}_2$  for some linear operator $\mathcal{T}_1$ and 
    \[
    \mathcal{T}_2
    :=
    (\id \otimes \id\otimes m\otimes \id\otimes \id)\circ 
    \delta^{-2}(\star\star).\]
    Next, we apply Equation~\ref{eq:Snake_Lemma} twice to $\mathcal{T}_2$ to obtain the third claim. We have 
    \begin{align*}
    \mathcal{T}_2
    &=
    (\id \otimes \id \otimes m \otimes \id \otimes \id)
    \circ
    [ (\id\otimes m^*)\delta^{-2}(\id\otimes A)m^* \otimes (m^*\otimes \id)(A\otimes \id)m^*]\\
    &=
   (\id \otimes [(\id \otimes m \otimes \id)\circ (m^*\otimes m^*)] \otimes \id) \circ 
    [\delta^{-2}(\id\otimes A)m^* \otimes (A\otimes \id)m^*]\\
    &\overset{(\ref{eq:Snake_Lemma})}{=}
    (\id \otimes [(m^*\otimes \id)m^*m]\otimes \id)\circ 
    [\delta^{-2}(\id\otimes A)m^* \otimes (A\otimes \id)m^*]\\
  &\overset{(\ref{eq:Snake_Lemma})}{=}
    (\id \otimes (m^*\otimes \id)m^*\otimes \id)\circ 
    \underbrace{(\id \otimes m\otimes id)[\delta^{-2}(\id\otimes A)m^* \otimes (A\otimes \id)m^*]}_{=:\mathcal{T}_3}.
    \end{align*}
Hence, for any $1\leq a,b\leq d,$ and $1\leq i,j\leq n_a,1\leq k,\ell\leq n_b,$ 
\begin{align*}
    \mathcal{T}_3(f_{ij}^{(a)}\otimes f_{\ell k}^{(b)})
    &=
     (\id \otimes m\otimes \id)\left(\delta^{-2}(\id \otimes A)m^*(f_{ij}^{(a)})\otimes (A\otimes \id)m^*(f_{\ell k}^{(b)})\right)\\
    &= (\id \otimes m\otimes \id)
    \left(f_{ij}^{(a)}\cdot \epsilon_{\cG} \otimes \sum_{m=1}^{n_b} A(f_{\ell m}^{(b)})\otimes f_{mk}^{(b)}\right)\\
    &=
    \left(\sum_{m=1}^{n_b} f_{ij}^{(a)}\cdot \epsilon_{\cG} \cdot A(f_{\ell m}^{(b)})\right)\otimes f_{mk}^{(b)}.
\end{align*}
If $f_{ij}^{(a)}\cdot \epsilon_{\cG}\cdot A(f_{\ell m}^{(b)})=0$ for all $1\leq m\leq n_b$, then $\mathcal{T}_3(f_{ij}^{(a)}\otimes f_{\ell k}^{(b)}) = 0$. As $\mu(S\otimes S^*) = \widetilde{\mathcal{W}}\circ \mathcal{T}_3$ for some linear operator $\widetilde{\mathcal{W}}$, we have $\mu(S\otimes S^*)(f_{ij}^{(a)}\otimes f_{\ell k}^{(b)}) = 0$ whenever $f^{(a)}_{ij}\cdot \epsilon_\cG\cdot A(f_{\ell m}^{(b)}) = 0 $ for all $1\leq m\leq n_b$. The contrapositive yields the final claim.
\end{proof}




\section{Cuntz algebras from dephasing quantum graphs}\label{sec:cuntz_algebra_embedding_theorem}

In this section, we apply Theorem~\ref{thm:lqck_structure_theorem} to the dephasing graph. Let $\psi$ be a faithful $\delta$-form on $B$, and define the {\em dephasing graph} $D:B\to B$ by
\[ 
D(f_{ij}^{(a)}):=\delta_{i=j} \psi(e_{ii}^{(a)})\delta^2 f_{ii}^{(a)}
\quad \forall f_{ij}^{(a)} \in B.
\]
For all $1\leq a\leq d,1\leq i,j\leq n_a$, observe
\begin{align*}
m(D\otimes D)m^*(f^{(a)}_{ij})
&=
\sum_{k=1}^{n_a} D(f_{ik}^{(a)})D(f_{kj}^{(a)})\\
&=\delta_{i=j}\left[\delta^2 \psi(e_{ii}^{(a)}) f_{ii}^{(a)} \delta^2 \psi(e_{ii}^{(a)})f_{ii}^{(a)}\right]\\
&=\delta^2 \left[\delta_{i=j} \delta^2 \psi(e_{ii}^{(a)})f_{ii}^{(a)}\right]\\
&= \delta^2 D(f_{ij}^{(a)}).
\end{align*}
By linearity, $m  (D\otimes D)  m^* = \delta^2 D$, so $D$ is a quantum adjacency matrix on $(B,\psi)$. Note also that the collection $\{K_i^{(a)}:1\leq a \leq d,1\leq i\leq n_a\}$ where 
\[K_i^{(a)}:=\delta \psi(e_{ii}^{(a)})^{3/2} \cdot f_{ii}^{(a)}\] is a linearly independent set of Kraus operators for $D$. Indeed, plugging in $f_{k\ell}^{(b)}$, we find
\begin{align*}
\sum_{a=1}^d \sum_{i=1}^{n_a} K_i^{(a)}f_{k\ell}^{(b)}(K_i^{(a)})^*
&=\sum_{a=1}^d\sum_{i=1}^{n_a} (\delta\psi(e_{ii}^{(a)})^{3/2}f_{ii}^{(a)} )f_{k\ell}^{(b)} (\delta\psi(e_{ii}^{(a)})^{3/2}f_{ii}^{(a)})^*\\
&=
\delta^2\sum_{a=1}^d\sum_{i=1}^{n_a} \psi(e_{ii}^{(a)})^3 f_{ii}^{(a)} f_{k\ell}^{(b)}f_{ii}^{(a)}\\
&=
\delta_{k=\ell}\delta^2 \psi(e_{ii}^{(a)})^{3 }\psi(e_{ii}^{(a)})^{-2}f_{kk}^{(b)}\\
&=
\delta_{k=\ell}\delta^2 \psi(e_{ii}^{(a)})f_{kk}^{(b)}\\
&=
D(f_{k\ell}^{(b)}).
\end{align*}
Therefore, $D$ is cp and has Kraus rank $r(D)=\sum_{a=1}^d n_a$. If it held for arbitrary $B$, Theorem~\ref{thm:morita_equivalence} would yield $\cO_{r(D)}$ as the stable isomorphism class of $L\cO_(B,\psi,D)$. In the equal-size case $B=\oplus_{a=1}^d M_n$, we predict that the local quantum Cuntz--Krieger algebra is not just stably isomorphic to $\cO_{r(D)}$, but isomorphic to $\cO_{r(D)}.$


\begin{example}

For the single-vertex dephasing graph $(M_n,\tr,D)$, there is a minimal Kraus decomposition of $D$ given by $\{n f_{ii}: 1\leq i\leq n\}\subset M_n(\mathbb{C})$. Hence, Theorem~\ref{thm:morita_equivalence} implies that $L\cO(M_n,\tr,D)$ is stably isomorphic to $\cO_n$. This theorem has not been established in a more general setting beyond $B=M_n$.

\end{example}


To apply Theorem~\ref{thm:lqck_structure_theorem} to the dephasing graph, we need to know about the quantum edge correspondence $E_D$ for a dephasing graph $(B,\psi,D)$. Its quantum edge checker is given by
\[
\epsilon_D
=
\delta^{-2} (\id \otimes D)m^*(1)
=
\sum_{a=1}^d\sum_{j=1}^{n_a} \psi(e_{jj}^{(a)})^2 f_{jj}^{(a)}\otimes f_{jj}^{(a)},
\]
and so acting on the left and right of $\epsilon_D$ by matrix units, we have that
\[
E_D
=
\spn{\{f^{(a)}_{ij}\otimes f_{jk}^{(a)}: 1\leq a\leq d,1\leq i,j,k\leq n_a\}}.
\]
In particular, any pair of matrix units $f_{ij}^{(a)}$, $f_{\ell m}^{(b)}\in B$ satisfy
\begin{equation}\label{eq:cyclic_generator_dephasing}
f_{ij}^{(a)}\cdot \epsilon_D \cdot f_{\ell m}^{(b)}
\neq 0
\quad \Longrightarrow \quad
a=b\text{ and }j=\ell.
\end{equation}
We shall use this fact in the proof of Theorem~\ref{thm:unital_embedding_of_O_n}, but first, we state and prove Proposition~\ref{prop:LQCK_relations_for_dephasing}, which gives a characterization of the generators of a local quantum Cuntz--Krieger $(B,\psi,D)$-family in terms of renormalized generators. We assume $\psi$ is a faithful $\delta$-form, so that, in particular, $\psi(e_{jj}^{(a)})^{-1}$ is well defined for any choice of $1\leq a\leq d$ and $1\leq j\leq n_a$.


\begin{proposition}\label{prop:LQCK_relations_for_dephasing}
    Let $B\cong \oplus_{a=1}^d M_{n_a}$ and let $\mathcal{C}$ be a unital $C^*$-algebra. Given a linear map $S:B\to \mathcal{C}$, for each $1\leq a \leq d$ and $1\leq i,j\leq n_a$, set 
    \[
    T^{(a)}_{ij}
    :=\delta \,\psi(e_{ii}^{(a)})^{1/2}\cdot S(f^{(a)}_{ij}).
    \] Then $S$ is a local quantum Cuntz--Krieger $(B,\psi,D)$-family if and only if    \begin{equation}\label{eq:LQCK1}\displaystyle \sum\limits_{m=1}^{n_a} T^{(a)}_{im}(T^{(a)}_{jm})^*T^{(b)}_{k\ell}= \delta_{a=b}\delta_{j=k}\, T^{(a)}_{i\ell}\end{equation}
    \begin{equation}\label{eq:LQCK2}
    \displaystyle (T^{(a)}_{ji})^*T^{(b)}_{k\ell} 
    = \delta_{a=b}\delta_{\stackrel{j=k}{i=\ell}} \left(\sum_{m=1}^{n_a} T^{(a)}_{im}(T^{(a)}_{im})^*\right)
    \end{equation} 
    \begin{equation}\label{eq:LQCK3}
    \displaystyle \sum_{a=1}^d \sum\limits_{i,j=1}^{n_a} T^{(a)}_{ij}(T^{(a)}_{ij})^*
    =1_{\mathcal{C}}
    \end{equation}
for all $1\leq a,b\leq d$ and $1\leq i,j\leq n_a, 1\leq k,\, \ell\leq n_b$.
\end{proposition}


\begin{proof}
Equations \ref{eq:LQCK1} and \ref{eq:LQCK3} are straightforward to verify. 
   For $f^{(a)}_{ij}\otimes f^{(b)}_{k\ell}\in B\otimes B$, observe:
    \begin{align*}
        (S^{(a)}_{ji})^* S^{(b)}_{k\ell} 
        &=\mu  (S^*\otimes S)(f^{(a)}_{ij}\otimes f_{k\ell}^{(b)})\\
         &= \displaystyle \delta^{-2} [\mu  (S\otimes S^*)  m^*  D  m](f^{(a)}_{ij}\otimes f_{k\ell}^{(b)}) \tag*{(\ref{eq:lqck2})}\\
        &= \delta_{a=b}\delta_{j=k}\left(\displaystyle \delta^{-2}\,  \,[\mu  (S\otimes S^*)  m^*  D](\psi(e_{jj}^{(a)})^{-1}f^{(a)}_{i\ell})\right)\\
        &= \delta_{a=b}\delta_{\stackrel{j=k}{i=\ell}} \left(\displaystyle \frac{1}{\delta^2 \psi(e_{jj}^{(a)})}\, \,[\mu  (S\otimes S^*)  m^*]\left(\delta^2 \psi(e_{ii}^{(a)})f_{ii}^{(a)}\right)\right)\\
       &=\delta_{a=b} \delta_{\stackrel{j=k}{i=\ell}}  \left(\displaystyle \frac{\psi(e_{ii}^{(a)})}{\psi(e_{jj}^{(a)})} \,[\mu  (S\otimes S^*)  m^*](f^{(a)}_{ii}) \right) \\
        & = \delta_{a=b} \delta_{\stackrel{j=k}{i=\ell}}\left(\displaystyle \frac{\psi(e_{ii}^{(a)})}{\psi(e_{jj}^{(a)})} \,\sum_{p=1}^{n_a}S^{(a)}_{ip}(S^{(a)}_{ip})^*\right).
    \end{align*}
    Therefore,
    \begin{align*}
        (T_{ji}^{(a)})^* T_{k\ell}^{(b)} 
   & = \delta^2 \psi(e_{jj}^{(a)}) \delta_{a=b} \delta_{\stackrel{j=k}{i=\ell}}\left(\displaystyle \frac{\psi(e_{ii}^{(a)})}{\psi(e_{jj}^{(a)})} \,\sum_{m=1}^{n_a}S^{(a)}_{im}(S^{(a)}_{im})^*\right)\\
&    =
    \delta_{a=b} \delta_{\stackrel{j=k}{i=\ell}}\left(\displaystyle \,\sum_{m=1}^{n_a}T^{(a)}_{im}(T^{(a)}_{im})^*\right).
    \end{align*}
\end{proof}


We arrive at our two main results of this section. For a general dephasing graph $(B,\psi,D)$, the identification of Cuntz isometries we use in the present paper requires that $B$ is isomorphic to a direct sum of copies of the same dimension matrix algebra, i.e., $B \cong \oplus_{a=1}^d M_n$, where $n$ is fixed for all $1\leq a\leq d$. In this case, we obtain a unital embedding of $\cO_n$ inside $L\cO(B,\psi,D)$, independent of choice of $d$ or $\psi$.


\begin{theorem}\label{thm:unital_embedding_of_O_n}
     Suppose $B\cong \oplus_{a=1}^d M_n$ for some fixed $n\in \mathbb{N}$, and let $s:B\to L\cO(B,\psi,D)$ denote the universal local quantum Cuntz--Krieger $(B,\psi,D)$-family. For each $1\leq a\leq d$ and $1\leq i,j \leq n_a$, set 
     \[
     t^{(a)}_{ij}
     :=
     \delta\psi(e_{ii}^{(a)})^{1/2}\cdot s(f_{ij}^{(a)}).
     \] 
     and 
      \[
     \forall 1\leq i\leq n: \quad
     V_i
     :=\sum_{a=1}^d\sum_{j=1}^n t^{(a)}_{ij}.
     \] 
     Then the collection $\{V_i:1\leq i\leq n\}$ 
    consists of Cuntz isometries which generate a unital embedding of $\cO_n$ inside $L\cO(B,\psi,D).$
\end{theorem}


\begin{proof}
    Suppose $B\cong \oplus_{a=1}^d M_n$ and let $s:B\to L\cO(B,\psi,D)$ be the universal local quantum Cuntz--Krieger family with unit ${\bf 1}:={\bf 1}_{L\cO(B,\psi,D)}$. Fix $1\leq a,b\leq d$ and $1\leq i,j\leq n_a$, $1\leq k,\ell\leq b$. We compute
    \[
        \left(\sum_{a=1}^d\sum_{j=1}^n (t^{(a)}_{ij})^*\right)\left(\sum_{b=1}^d \sum_{k=1}^n t^{(b)}_{ik} \right)
        \overset{(\ref{thm:lqck_structure_theorem})}{=}
        \sum_{a=1}^d \sum_{j=1}^n (t^{(a)}_{ij})^*t^{(a)}_{ij}  
        \overset{(\ref{eq:LQCK2})}{=}
        \sum_{a=1}^d \sum_{j,p=1}^n t^{(a)}_{jp}(t^{(a)}_{jp})^* 
        \overset{(\ref{eq:LQCK3})}{=}
        {\bf 1}.
    \]
   Therefore, $V_i^*V_i = {\bf 1}$ for each $1\leq i\leq n$. Note $V_i^*V_\ell = 0$ whenever $i\neq \ell$ by Equation~\ref{eq:LQCK2} and the definition of the $V_i$, so the ranges of $V_iV_i^*$ and $V_\ell V_\ell^*$ are orthogonal. 
   
   Next we show $\sum_{i=1}^n V_iV_i={\bf 1}.$ By definition of $D$ and Equation~\ref{eq:cyclic_generator_dephasing}, for all $1\leq m\leq n_b,$
   \[
   f_{ij}^{(a)}\cdot \epsilon_D \cdot D(f_{\ell m}^{(a)})
   =
   f_{ij}^{(a)}\cdot \epsilon_D \cdot \delta_{\ell = m} \psi(e_{\ell \ell}^{(b)}) \delta^2 f_{\ell \ell}^{(b)}
   =
   \delta_{\stackrel{a=b}{j=\ell} }(\psi(e_{jj}^{(a)}) \delta ^2(f_{ij}^{(a)}\cdot \epsilon_D \cdot f_{jj}^{(a)})).
   \]
   Hence, Theorem~\ref{thm:lqck_structure_theorem} implies that the terms $t_{ij}^{(a)}(t_{ik}^{(b)})^*$ appearing in $V_iV_i^*$ must equal $0$ unless $a=b$ and $j=k$. This yields
    \[
        \sum\limits_{i=1}^n
        V_iV_i^*
        =
        \sum\limits_{a,b=1}^d\sum\limits_{i,j,k=1}^n t^{(a)}_{ij}  (t^{(b)}_{ik})^* 
        =
        \sum\limits_{a=1}^d\sum\limits_{i,j=1}^n t^{(a)}_{ij}  (t^{(a)}_{ij})^* 
        \overset{(\ref{eq:LQCK3})}{=}{\bf 1}.
    \]
Therefore, $\{V_i:1\leq i\leq n\}$ consists of Cuntz isometries, and by simplicity of $\cO_n$, we have that $C^*(\{V_i:1\leq i\leq n\})$ is isomorphic to $\cO_n$. We conclude that $\cO_n$ embeds unitally into $L\cO(B,\psi,D)$ whenever $B\cong \oplus_{a=1}^d M_n.$ 
\end{proof}



The unital embedding of $\cO_n$ inside $L\cO(\oplus_{a=1}^d M_n,\psi,D)$ could provide insight to the isomorphism class of $L\cO(\oplus_{a=1}^d M_n,\psi,D)$ if a generalization of Theorem~\ref{thm:morita_equivalence}, as the following, were to hold: {\em Let $B\cong \oplus_{a=1}^d M_n$. Given a cp map $A:B\to B$, is the associated Cuntz--Pimsner algebra $\cO_{B\otimes_A B}$, as in \cite{Marrero-Muhly05}, Morita equivalent to the Cuntz algebra on $r(A)$ generators?} If true, we could conclude that for an arbitrary finite quantum graph $(B,\psi,A)$ such that $A$ is cp and the kernel of $A$ contains no central summand of $B$, the local quantum Cuntz--Krieger algebra $L\cO(B,\psi,A)$ is stably isomorphic to the Cuntz algebra on $r(A)$ generators. In particular, this would imply that $L\cO(\oplus_{a=1}^d M_n,\psi,D)$ is stably isomorphic to $\cO_{dn}$, the Cuntz algebra on $dn$ generators.

We conclude with the main application of Theorem~\ref{thm:lqck_structure_theorem}.

\begin{corollary}\label{cor:cuntz_isometries_for_L(M_n-tr-D}
    Let $s:M_n\to L\cO(M_n,\tr,D)$ be the universal local quantum Cuntz--Krieger $(M_n,\tr,D)$-family. For each $1\leq i,j \leq n$, set \[t_{ij}:=\sqrt{n}\cdot s(f_{ij}).\] Then the collection $\{V_i:1\leq i\leq n\}$ where 
     \[
     \forall 1\leq i\leq n: \quad
     V_i
     :=\sum_{j=1}^n t_{ij}
     \] consists of Cuntz isometries which generate all of $L\cO(M_n,\tr,D).$
\end{corollary}
    
\begin{proof}
Note that $\sqrt{n} = \delta \psi(e_{ii})^{1/2}$ when $\psi=\tr$. Hence, Theorem~\ref{thm:unital_embedding_of_O_n} confirms $\{V_i:1\leq i \leq n\}$ are Cuntz isometries, so it remains to show that $C^*(\{V_i:1\leq i \leq n\})=L\cO(M_n,\tr,D).$ For any $1\leq i,j,k,\ell\leq n$, as $t_{ij}t_{k\ell}\neq 0\iff 
S_{ij}S_{k\ell} \neq 0$, Theorem~\ref{thm:lqck_structure_theorem} yields
\[
t_{ij}t_{k\ell}\neq 0 
\quad
\Longrightarrow 
\quad
f_{ij}\cdot \epsilon_D \cdot f_{k\ell}\neq 0 
\quad
{\Longrightarrow}
\quad
j=k,
\]
which tells us $t_{ik}V_j = \sum_{\ell=1}^d t_{ik}t_{j\ell} = 0$ unless $j=k.$ Hence,
\[V_iV_j
=\sum_{k,\ell=1}^n t_{ik}t_{j\ell}
=\sum_{\ell=1}^n t_{ij}t_{j\ell}
=t_{ij}\left(\sum_{\ell=1}^n t_{j\ell}\right)
=t_{ij}V_j.
\]
Using 
$
{\bf 1}
=
\sum_{k=1}^d V_kV_k^*
$
from Theorem~\ref{thm:unital_embedding_of_O_n}, we have
\[
    t_{ij}
    =t_{ij}{\bf 1}
    =t_{ij}\left(\sum_{k=1}^nV_kV_k^*\right)
    =\sum_{k=1}^n t_{ij}V_kV_k^*
    =t_{ij}V_jV_j^*
    =V_iV_jV_j^*.
\] 
Therefore, all of the $t_{ij}$ belong to $C^*(\{V_i:1\leq i\leq n\})$, and because the $t_{ij}$ generate $L\cO(M_n,\tr,D)$, we obtain
$L\cO(M_n,\tr,D)
\cong \cO_n$.%
\end{proof}

\section*{Acknowledgements}
We wish to thank M. Brannan, M. Hamidi, B. Nelson, and M. Wasilewski for their comments on a near-final draft of the paper and for helpful discussion at the Institut-Mittag Leffler Workshop on Operator Algebras and Quantum Information. We wish to thank C. Schafhauser for discussing future directions of applying classification tools to distinguish local and full quantum Cuntz--Krieger algebras, as well as J. Crann, L. Gao, I. Todorov, and L. Turowska for discussion which motivated the calculation of the local quantum Cuntz--Krieger algebra for the dephasing graph at the American Institute for Mathematics. We also wish to thank the Institut Mittag-Leffler and the American Institute of Mathematics. The author was supported by National Science Foundation Division of Mathematical Sciences grant 2603107.

\bibliographystyle{alpha}
\bibliography{final_arxiv_1}
\end{document}